\documentclass{amsart}
\pdfoutput=1
\usepackage{amsfonts}
\usepackage{mathrsfs}
\usepackage{amssymb}
\usepackage{amsmath}
\usepackage{amsrefs}
\usepackage{centernot}
\usepackage{mathtools}

\newtheorem{theorem}{Theorem}[section]
\newtheorem{lemma}[theorem]{Lemma}
\newtheorem{proposition}[theorem]{Proposition}

\theoremstyle{definition}
\newtheorem{definition}[theorem]{Definition}

\theoremstyle{remark}
\newtheorem{remark}[theorem]{Remark}

\allowdisplaybreaks

\newcommand{\CC}{\mathbb{C}}
\newcommand{\RR}{\mathbb{R}}

\newcommand{\spn}{\mathrm{span}}
\newcommand{\V}{\mathcal V}

\DeclarePairedDelimiter{\ceil}{\lceil}{\rceil}

\begin{document}

\title{The quantum Ramsey numbers $QR(2,k)$}
\author{Andrew Allen}
\address{Mathematics and Statistics, Dalhousie University, Halifax, Nova Scotia}
\email{an847773@dal.ca}
\author{Andre Kornell}
\address{Mathematical Sciences, New Mexico State University, Las Cruces, New Mexico}
\email{kornell@nmsu.edu}
\thanks{This work was supported by the Air Force Office of Scientific Research under Award No.~FA9550-21-1-0041 and by the National Science Foundation under Award No.~DMS-2231414.}
\subjclass[2020]{}

\begin{abstract}
Operator systems of matrices can be viewed as quantum analogues of finite graphs. This analogy suggests many natural combinatorial questions in linear algebra. We determine the quantum Ramsey numbers $QR(2,k)$ and the lower quantum Tur\'{a}n numbers $T^\downarrow(n, m)$ with $m \geq n/4$. In particular, we conclude that $QR(2,2) = 4$ and confirm Weaver's conjecture that $T^\downarrow(4, 1) = 4$. We also obtain a new result for the existence of anticliques in quantum graphs of low dimension.
\end{abstract}

\maketitle

\section{Introduction}

A \emph{quantum graph} on the von Neumann algebra $M_n := M_n(\CC)$ is a subspace $\V \subseteq M_n$ such that $I_n \in \V$ and such that $A^* \in \V$ for all $A \in \V$ \cite{DuanSeveriniWinter}*{Section~II}. In other words, it is an operator system in $M_n$ \cite{ChoiEffros}*{p.~157}. Quantum graphs were first introduced in \cite{DuanSeveriniWinter} as a quantum analogue of confusability graphs of classical information channels. This notion was then generalized to arbitrary von Neumann algebras \cites{KuperbergWeaver, Weaver}. A quantum graph on $\CC^n$ is essentially just a simple graph on the vertex set $\{1, \ldots, n\}$, with the convention that each vertex is adjacent to itself.

This paper concerns an analogue of Ramsey's theorem for finite simple graphs.

\begin{theorem}[\cite{Ramsey}*{Theorem~B}]\label{Ramsey}
Let $j$ and $k$ be positive integers. Then, there exists a positive integer $N$ such that, for all $n \geq N$, each simple graph $G$ on the set $[n]= \{1, \ldots, n\}$ has a $j$-clique or a $k$-anticlique.
\end{theorem}

\noindent Recall that a \emph{clique} is a subset of $[n]$ such that each pair of distinct elements is adjacent and that an \emph{anticlique} is a subset of $[n]$ such that each pair of distinct elements is not adjacent. A \emph{$j$-clique} is a clique of cardinality $j$, and a \emph{$k$-anticlique} is an anticlique of cardinality $k$. The minimum positive integer $N$ that satisfies the conclusion of Theorem~\ref{Ramsey} is the \emph{Ramsey number} $R(j,k)$. The determination of Ramsey numbers is a difficult problem even for small parameters; most recently, McKay and Radziszowski showed that $R(4,5) = 25$  \cite{McKayRadziszowski}.

Weaver proved the following quantum analogue of Theorem~\ref{Ramsey}.

\begin{theorem}[\cite{Weaver2}*{Theorem~4.5}]\label{WeaverRamsey}
Let $j$ and $k$ be positive integers. Then, there exists a positive integer $N$ such that, for all $n \geq N$, each quantum graph $\V$ on the von Neumann algebra $M_n$ has a $j$-clique or a $k$-anticlique.
\end{theorem}

\noindent In this context, a \emph{clique} is a projection $P \in M_n$ such that $P \V P = P M_n P$, and an \emph{anticlique} is a projection $P \in M_n$ such that $P \V P = \spn\{P\}$. A $j$-clique is a clique of rank $j$, and a $k$-anticlique is an anticlique of rank $k$. The minimum positive integer $N$ that satisfies the conclusion of Theorem~\ref{WeaverRamsey} is the \emph{quantum Ramsey number} $QR(j,k)$.

It is immediate that $QR(j, 1) = 1$ and $QR(1, k) = 1$ for all positive integers $j$ and $k$ because $P \V P = \spn\{P\}$ for any rank-$1$ projection $P$. The other quantum Ramsey numbers are nontrivial, and to our knowledge, none were known prior to this work. We prove that $QR(2,2) = 4$ and, more generally, that $QR(2, k) = 3k - 2$ for every positive integer $k$. We argue that $\V$ has a $2$-clique if $\dim(\V) \geq 4$, as observed by Weaver \cite{Weaver2}, and that $\V$ has a $k$-anticlique if $\dim(\V) \leq 3$, as a consequence of a theorem of Li, Poon, and Sze on the higher-rank numerical range of a matrix \cite{LiPoonSze}. We do not determine the quantum Ramsey numbers $QR(j, 2)$ for arbitrary positive integers $j$; the familiar symmetry between cliques and anticliques does not occur for quantum graphs.

The guaranteed existence of a $k$-anticlique when $\dim(\V) \leq 3$ and $n \geq 3k - 2$ can be expressed as $T^\downarrow(n,m) \geq 4$ for $n > 3m$. The \emph{lower quantum Tur\'{a}n number} $T^\downarrow(n,m)$ is the smallest integer $d$ such that there exists an operator system $\V \subseteq M_n$ with $\dim \V = d$ and with no $(m+1)$-anticliques \cite{Weaver3}*{sec.~2}. Implicitly, $n > m$. The lower quantum Tur\'{a}n number is an indirect quantum analogue of the Tur\'{a}n number, which is the number of edges in a Tur\'{a}n graph \cite{Turan}.

Combining the inequality $T^\downarrow(n,m) \geq 4$ for $n > 3m$ with \cite{Weaver3}*{Theorem~2.11}, we find that $T^\downarrow(n,m) = 4$ for $3m < n \leq 4m$. This confirms Weaver's conjecture that $T^\downarrow(4, 1) = 4$ \cite{Weaver3}*{p.~343}. More generally, this confirms Weaver's conjecture that $T^\downarrow(n,m) = \ceil{\frac n m}$ in the range $m \geq n /4$ since it holds in the range $m \geq n / 3$ by \cite{Weaver3}*{Theorem~2.10} and \cite{Weaver3}*{Theorem~2.11}.

Finally, we establish a new lower bound on $n$ that guarantees the existence of a $k$-anticlique in a $d$-dimensional quantum graph $\V \subseteq M_n(\CC)$. This new lower bound improves on \cite{Weaver3}*{Theorem~2.10} when $d$ is small. The existence of anticliques is significant to quantum information theory because the anticliques in the quantum confusability graph of a quantum channel are exactly the zero-error quantum error-correcting codes for that quantum channel; see \cite{DuanSeveriniWinter}*{sec.~III} and \cite{KnillLaflamme}*{sec.~3}. Indeed, this equivalence was a core motivation for the introduction of quantum graphs \cite{DuanSeveriniWinter}.

\section{Results}

We first prove that every $3$-dimensional quantum graph $\V \subseteq M_n$ has a $k$-anticlique whenever $n \geq 3k - 2$. To do so, we apply a theorem of Li, Poon, and Sze on the higher numerical range of a matrix \cite{LiPoonSze}. This notion was introduced by Choi, Kribs, and \.{Z}yczkowski \cite{ChoiKribsZyczkowski}.

\begin{definition}[\cite{ChoiKribsZyczkowski}*{Equation~1}]
The \emph{rank-$k$ numerical range} of a matrix $A \in M_n$ is the set
$
\Lambda_k(A) = \{\lambda \in \CC \mid PAP = \lambda P \text{ for some rank-$k$ projection }P\}.
$
\end{definition}
\noindent Note that the rank-$1$ numerical range of $A$ is its numerical range in the usual sense.

Choi, Giesinger, Holbrook, and Kribs conjectured that $\Lambda_2(A) \neq \emptyset$ for all $A \in M_4$ \cite{ChoiGiesingerHolbrookKribs}*{Remark~2.10}. Li, Poon, and Sze confirmed this conjecture, proving the following.

\begin{theorem}[\cite{LiPoonSze}*{Theorem~1}]\label{LiPoonSze}
If $n \geq 3k -2$, then $\Lambda_k(A) \neq \emptyset$ for all $A \in M_n$.
\end{theorem}

We find it more convenient to work with isometries rather than projections. Thus, we reformulate this theorem to say that, if $n \geq 3k -2$, then for each matrix $A \in M_n$, there exists an isometry $J \in M_{n \times k}$ such that $J^\dagger A J \in \mathrm{span}\{I_k\}$. Similarly, we observe that a quantum graph $\V \subseteq M_n$ has a $k$-anticlique iff $J^\dagger \V J = \mathrm{span}\{I_k\}$ for some isometry $J \in M_{n \times k}$.

\begin{lemma}\label{lemma}
Let $\V \subseteq M_n$ be a quantum graph with $\dim \V \leq 3$. If $n \geq  3k -2$, then $\V$ has a $k$-anticlique.
\end{lemma}

\begin{proof}
Assume that $n \geq  3k -2$. Since $\V$ is an operator system with $\dim(\V) \leq 3$, $\V = \spn\{I_n, A_1, A_2\}$ for some Hermitian matrices $A_1, A_2 \in M_n$. Let $A = A_{1} + i A_{2}$.
By Theorem~\ref{LiPoonSze}, there exist an isometry $J \in M_{n \times k}$ and $\lambda_1, \lambda_2 \in \RR$ such that
$$
J^\dagger A_1 J + i J^\dagger A_2 J = J^\dagger A J = (\lambda_1 + i \lambda_2) I_k = \lambda_1 I_k + i \lambda_2 I_k,
$$
and thus such that 
$$
J^\dagger I_n J = I_k, \qquad J^\dagger A_1 J = \lambda_1 I_k, \qquad J^\dagger A_2 J = \lambda_2 I_k.
$$
Therefore, if $n \geq  3k -2$, then $\V$ has a $k$-anticlique.
\end{proof}

We now use this lemma to confirm Weaver's conjecture that $T^\downarrow(n,m) = \ceil{\frac n m}$ for a new range of parameters, establishing the value of $T^\downarrow(4,1)$ \cite{Weaver3}*{p.~343}.

\begin{theorem}
If $3m < n \leq 4m$, then $T^\downarrow(n,m) = 4$.
\end{theorem}

\begin{proof}
We have that $T^\downarrow(n,m) \geq 4$ by Lemma~\ref{lemma} and that $T^\downarrow(n,m) \leq \ceil{\frac n m} \leq 4$ by \cite{Weaver3}*{Theorem~2.11}.
\end{proof}

In order to compute the quantum Ramsey numbers $QR(2,k)$ for $k\geq2$, we first exhibit a quantum graph that has neither a $2$-clique nor a $k$-anticlique.

\begin{proposition}\label{proposition}
The quantum graph $\V = \spn\{A_1, A_2, A_3\}$, where 
$$
A_1 = \left( \begin{matrix} I_{k-1} & 0 & 0 \\ 0 & 0 & 0 \\ 0 & 0 & 0 \end{matrix}\right), \quad  A_2 = \left( \begin{matrix} 0 & 0 & 0 \\ 0 & I_{k-1} & 0 \\ 0 & 0 & 0 \end{matrix}\right), \quad A_3 = \left( \begin{matrix} 0 & 0 & 0 \\ 0 & 0 & 0 \\ 0 & 0 & I_{k-1}\end{matrix}\right),
$$
has no $2$-cliques and no $k$-anticliques. Note that $\V \subseteq M_n$, where $n = 3k -3$.
\end{proposition}

\begin{proof}
The quantum graph $\V$ cannot have a $2$-clique because $\dim(\V) < 4$, and it cannot have a $k$-anticlique by \cite{Weaver3}*{Proposition~2.1}.
\end{proof}

\begin{theorem}
For all positive integers $k$, we have that $QR(2,k) = 3k-2$.
\end{theorem}

\begin{proof}
The $k=1$ case is immediate. For $k\geq2$, we have that $QR(2,k) > 3k-3$ by Proposition~\ref{proposition}. It remains to show that $QR(2,k) \leq 3k-2$. Let $n \geq 3k-2$. For each quantum graph $\V \subseteq M_n$, if $\dim(\V) \leq 3$, then $\V$ has a $k$-anticlique by Lemma~\ref{lemma}, and if $\dim(\V) \geq 4$, then $\V$ has a $2$-clique by \cite{Weaver2}*{Theorem~3.3}.
\end{proof}

We conclude with a generalization of Lemma~\ref{lemma}.

\begin{theorem}\label{bounds}
Let $\ell$ be a nonnegative integer. For all quantum graphs $\V \subseteq M_n$,
\begin{enumerate}
\item if $\dim(\V) \leq 2 \ell + 1$, then $\V$ has a $k$-anticlique whenever $n > 3^\ell(k-1)$,
\item if $\dim(\V) \leq 2 \ell + 2$, then $\V$ has a $k$-anticlique whenever $n > 2\cdot 3^\ell(k-1)$.
\end{enumerate}
\end{theorem}

\begin{proof}
We prove (1) for all positive integers $\ell$ by induction on $\ell$. The base case $\ell = 1$ is Lemma~\ref{lemma}. Assume that $\ell$ satisfies claim (1), and let $\V \subseteq M_m$ be a quantum graph such that $\dim(\V) \leq 2 (\ell+1) +1 = 2 \ell + 3$. Let $k$ be a positive integer such that $m > 3^{\ell+1}(k-1)$, and let $n = 3^\ell(k-1) + 1$. We have that
$$
m \geq 3^{\ell+1}(k-1) + 1 = 3 (3^\ell(k-1) + 1) - 2 = 3 n - 2.
$$

Since $\V$ is an operator system with $\dim(\V) \leq 2 \ell + 3$, $\V = \spn\{I_m, A_1, \ldots, A_{2\ell+ 2}\}$ for some Hermitian matrices $A_1, \ldots, A_{2\ell+2} \in M_m$. Let $A = A_{2\ell+1} + i A_{2\ell +2}$. From Theorem~\ref{LiPoonSze}, we obtain an isometry $J \in M_{m \times n}$ such that $J A J^\dagger \in \spn\{I_n\}$. Thus, $J^\dagger A_{2\ell+1} J$ and $J^\dagger A _{2\ell +2} J$ are both in the span of $I_n$. It follows that $J^\dagger \V J = \spn\{I_n, J^\dagger A_1 J, \ldots, J^\dagger A_{2\ell}J\}$, so $\dim(J^\dagger\V J) \leq 2\ell + 1$. By the induction hypothesis, there exists an isometry $K \in M_{n \times k}$ such that
$$
K^\dagger J^\dagger \V J K = \spn\{I_k\}.
$$
In other words, the isometry $JK$ provides a $k$-anticlique in $\V$. By induction, claim~(1) holds for all positive integers $\ell$. The $\ell = 0$ case is trivial.

The proof of claim (2) is the same apart from the base case. The base case $\ell = 0$ follows from \cite{Weaver3}*{Proposition~2.7} because
$$
\left\lceil \frac n 2\right \rceil \geq \left \lceil \frac{ 2 (k-1) +1 } 2 \right \rceil = \left \lceil (k - 1) + \frac 1 2 \right \rceil = (k - 1) + 1 = k.
$$
\end{proof}

\begin{remark}
Both Theorem~\ref{bounds} and \cite{Weaver3}*{Theorem~2.10} provide conditions on the integer $n$ that guarantee that a quantum graph $\V \subseteq M_n$ with $\dim(\V) = d$ has a $k$-anticlique. The latter theorem guarantees the existence of the anticlique if
$$
(k-1) d + 1 \leq \left \lceil \frac n {d-1} \right \rceil.
$$
It is clearly the stronger result for large $d$. However, both theorems provide sufficient lower bounds on $n$ that are approximately linear in $k$, and for small $d$, we obtain a better growth rate from Theorem~\ref{bounds}. For example, Theorem~\ref{bounds} guarantees the existence of a $k$-anticlique for all quantum graphs $\V \subseteq M_n$ with $\dim \V = 4$ for $n \geq 6k + o(k)$, while \cite{Weaver3}*{Theorem~2.10} guarantees the same conclusion only for $n \geq 12k + o(k)$.
\end{remark}

\section*{Acknowledgments}

We thank Julien Ross and Peter Selinger for helpful discussion.

\end{document}